\documentclass[11pt]{amsart}
\usepackage{amsmath,amsthm,amscd,amsfonts, amssymb, mathrsfs}
\usepackage{graphicx}
\usepackage{tikz}
\usepackage{color}
\include{diagxy}
\usepgflibrary{shapes}
\newcommand{\sect}[1]{\section{#1}\setcounter{equation}{0}}

\font\mbn=msbm10 scaled \magstep1
\font\mbs=msbm7 scaled \magstep1
\font\mbss=msbm5 scaled \magstep1
\newfam\mbff
\textfont\mbff=\mbn
\scriptfont\mbff=\mbs
\scriptscriptfont\mbff=\mbss

\newcommand{\Di}      {\mathbb{D}}

\newcommand\Co           {{\mathbb C}}

\newtheorem{Th}{Theorem}[section]
\newtheorem{Lm}[Th]{Lemma}
\newtheorem{C}[Th]{Corollary}

\newtheorem{R}[Th]{Remark}

\newtheorem*{Prop}{Proposition}
\newtheorem*{Theo}{Theorem}

\begin{document}
\title[Topology of the Maximal Ideal Space of $H^\infty$ Revisited]{Topology of the Maximal Ideal Space of $H^\infty$ Revisited}

\author{Alexander Brudnyi} 
\address{Department of Mathematics and Statistics\newline
\hspace*{1em} University of Calgary\newline
\hspace*{1em} Calgary, Alberta\newline
\hspace*{1em} T2N 1N4}
\email{albru@math.ucalgary.ca}

\thanks{Research supported in part by NSERC}

\begin{abstract}
Let $M(H^\infty)$ be the maximal ideal space of the Banach algebra $H^\infty$ of bounded holomorphic functions on the unit disk $\Di\subset\Co$. We prove that $M(H^\infty)$ is homeomorphic to the Freudenthal compactification $\gamma(M_a)$ of the set $M_a$ of 
all non-trivial (analytic disks) Gleason parts of $M(H^\infty)$. Also, we give alternative proofs of important results of Su\'{a}rez asserting that the set $M_s$ of trivial (one-pointed) Gleason parts of $M(H^\infty)$ is totally disconnected and that the \v{C}ech cohomology group $H^2(M(H^\infty),\mathbb Z)=0$.
 \end{abstract}

\date{}

\maketitle

\sect{Introduction}
The paper studies the topological structure of the maximal ideal space $M(H^\infty)$ of the Banach algebra $H^\infty$ of bounded holomorphic functions on the unit disc $\Di\subset\Co$ equipped with pointwise multiplication and supremum norm $\|\cdot\|_\infty$.
Recall that for a commutative unital complex Banach algebra $A$ with dual space $A^*$ the {\em maximal ideal space} $M(A)$ of $A$ is the set of nonzero homomorphisms $A\rightarrow\Co$ endowed with the {\em Gelfand topology}, the weak$^*$ topology induced by $A^*$. It is a compact Hausdorff space contained in the unit ball of $A^*$.
Let $C(M(A))$ be the algebra of
continuous complex-valued functions on $M(A)$ with supremum norm. The Gelfand transform $\hat \,: A\rightarrow C(M(A))$, defined by $\hat a(\varphi):=\varphi(a)$, is a nonincreasing-norm morphism of algebras that allows to thought of elements of $A$ as continuous functions on $M(A)$. 

In the case of $H^\infty$ evaluation at a point of $\Di$ is an element of $M(H^\infty)$, so $\Di$ is naturally embedded into $M(H^\infty)$ as an open subset. The famous Carleson corona theorem \cite{C} asserts that $\Di$ is dense in $M(H^\infty)$.

It is known that $M(H^\infty)$ is the union of two kinds of Gleason parts defined as follows.
Recall that the pseudohyperbolic metric on $\Di$ is given by
\[
\rho(z,w):=\left|\frac{z-w}{1-\bar w z}\right|,\qquad z,w\in\Di.
\]
For $x,y\in M(H^\infty)$ the formula
\[
\rho(x,y):=\sup\{|\hat f(y)|\, :\, f\in H^\infty,\, \hat f(x)=0,\, \|f\|_\infty\le 1\}
\]
extends $\rho$ to $M(H^\infty)$.
The {\em Gleason part} of $x\in M(H^\infty)$ is then determined by $\pi(x):=\{y\in M(H^\infty)\, :\, \rho(x,y)<1\}$. For $x,y\in M(H^\infty)$ we have
$\pi(x)=\pi(y)$ or $\pi(x)\cap\pi(y)=\emptyset$. Hoffman's classification of Gleason parts \cite{H} says that there are only two cases: either $\pi(x)=\{x\}$ or $\pi(x)$ is an analytic disk. The former case means that there is a continuous one-to-one and onto map $L_x:\Di\rightarrow\pi(x)$ such that
$\hat f\circ L_x\in H^\infty$ for every $f\in H^\infty$. Moreover, any analytic disk is contained in a Gleason part and any maximal (i.e., not contained in any other) analytic disk is a Gleason part. By $M_a$ and $M_s$ we denote sets of all non-trivial (analytic disks) and trivial (one-pointed) Gleason parts, respectively. It is known that $M_a\subset M(H^\infty)$ is open. Hoffman proved that $\pi(x)\subset M_a$ if and only if $x$ belongs to the closure of some $H^\infty$ interpolating sequence in $\Di$. 

More recent developments in the area are due to the work of Su\'{a}rez \cite{S1, S2} who proved the following profound results
\begin{itemize}
\item[(1)] The covering dimension of $M(H^\infty)$ is $2$;
\item[(2)] The second \v{C}ech cohomology group $H^2(M(H^\infty),\mathbb Z)=0$;
\item[(3)] The set of trivial Gleason parts $M_s$ is totally disconnected.
\end{itemize}
(Recall that for a normal space $X$, $\text{dim}\, X\le n$ if every finite open cover
of $X$ can be refined by an open cover whose order $\le n + 1$. If $\text{dim}\, X\le n$ and
the statement $\text{dim}\, X\le n-1$ is false, we say $\text{dim}\, X = n$.)

The original proof of property (1) in \cite{S1} is based on a deep result of Treil \cite{T} asserting that the Bass stable rank of $H^\infty$ is 1 along with some other powerful techniques of the theory of $H^\infty$.  An alternative proof, not using this fact but invoking property (3),  was given by the author in \cite{Br1}. Specifically, it was shown that the set of all non-trivial Gleason parts $M_a$ is homeomorphic to a fibre bundle over a compact Riemann surface $S$ of genus $g\ge 2$ with the fibre an open subset of the Stone-\v{C}ech compactification of the fundamental group of $S$. This implies that any compact subset of $M_a$ has covering dimension $\le 2$ which together with property (3) gives ${\rm dim}\, M(H^\infty)=2$ by a known topological result.

Property (2) is proved in \cite{S1} as one of important steps towards establishing property (1) in his conception. The proof relies completely on the above mentioned Treil's result \cite{T} and some constructions of this paper. In the present paper we show that property (2) is 
the consequence of the fact that $H^\infty$ is a projective free Banach algebra, see \cite[Cor.~3.30]{Q},  \cite[Th.~1.5]{BS}. The latter can be deduced from the classical Beurling-Lax-Halmos theorem (for its formulation see, e.g., \cite[p.~1025]{To} and references therein).

Property (3) is proved in \cite{S2} using some results from \cite{S1} and is based on a modification of the construction of Garnett and Nicolau \cite{GN}  who exploited it to show  that interpolating Blaschke products generate $H^\infty$. In this paper we give an alternative proof of property (3) based on the classical construction of Carleson \cite{C}.

Finally, we prove a result describing the topological nature of $M(H^\infty)$ asserting that
$M(H^\infty)$ is homeomorphic to the Freudenthal compactification $\gamma(M_a)$ (sometimes referred to as the end compactification, see \cite{F}, \cite{M}) of the set of all non-trivial Gleason parts $M_a$. Thus each trivial Gleason part is the end of $M_a$ in the sense of Freudenthal.

\sect{$M_s$ is Totally Disconnected} 
A topological space $X$ is {\em totally disconnected} if any subset of $X$ containing more than two points is disconnected. If $X$ is a compact Hausdorff space, then it is totally disconnected if and only if $\text{dim}\, X=0$ (see, e.g., \cite{N} for basic results of the dimension theory).

For a continuous function $g:X\rightarrow\Co$ we set $S_X(g;\varepsilon):=\{x\in X\, :\,  |g(x)|<\varepsilon\}$. By ${\rm cl}_X$ we denote closure in $X$.
In the next result, $\hat g\in C(M(H^\infty))$ stands for the (continuous) extension of  $g\in H^\infty$ to $M(H^\infty)$ by means of the Gelfand transform. Also, we equip the set of trivial Gleason parts $M_s\subset M(H^\infty)$ by the induced topology.

Let $f\in H^\infty\setminus\{0\}$, $\|f\|_\infty=1$, be such that $\hat f(x)=0$ for some $x\in M_s$.  Recall that for  each $\delta\in (0,1)$, the classical Carleson construction \cite{C} produces a positive $\varepsilon=\varepsilon(\delta)$ and an open set $\Omega_\varepsilon$ with the boundary $\Gamma_\varepsilon$ being a {\em Carleson contour} such that
\begin{equation}\label{eq1}
S_{\Di}(f;\varepsilon)\subset\Omega_\varepsilon\subset S_{\Di}(f;\delta/2).
\end{equation}

We have
\begin{Th}\label{te1}
 ${\rm cl}_{M(H^\infty)}(\Omega_\varepsilon)\cap M_s$ is a clopen subset of $S_{M_s}(\hat f;\delta)$ containing $x$.
 \end{Th}
 \begin{proof}
 It is well known (see, e.g., \cite[Ch.~VIII, Sect.\,4]{Ga}) that there is a $H^\infty$ interpolating sequence $\{z_n\}\subset\Gamma_{\varepsilon}$ and a number $c\in (0,1)$ such that
\begin{equation}\label{eq2}
\inf_j\rho(z_j,z)< c\quad {\rm for\ all}\quad z\in\Gamma_{\varepsilon}.
\end{equation}
Due to the result of Hoffman \cite{H}, ${\rm cl}_{M(H^\infty)}(\{z_n\})\subset M_a$. Then \eqref{eq2} implies that ${\rm cl}_{M(H^\infty)}(\Gamma_{\varepsilon})\subset M_a$ as well (see, e.g., \cite[Ch.~X]{Ga} ).

Next, consider the open set $U:=M(H^\infty)\setminus {\rm cl}_{M(H^\infty)}(\Gamma_{\varepsilon})$. By definition 
\[
U\cap\Di=\Omega_{\varepsilon}\sqcup \bigl(\Di\setminus {\rm cl}_{\Di}(\Omega_{\varepsilon})\bigr).
\]
Let $g: U\cap\Di\rightarrow\{0,1\}$ be the indicator function of $\Omega_{\varepsilon}$. Clearly, $g\in H^\infty(U\cap\Di)$. Hence, \cite[Th.~3.2]{S1} implies that $g$ admits a continuous extension $\tilde g\in C(U)$. Observe that
\[
U\cap M_s=\bigl(M(H^\infty)\setminus {\rm cl}_{M(H^\infty)}(\Gamma_{\varepsilon})\bigr)\cap M_s=M_s.
\]
So, $\tilde g|_{M_s}\in C(M_s)$ attains values $0$ and $1$ only. In particular, ${\rm cl}_{M(H^\infty)}(\Omega_{\varepsilon})\cap M_s=(\tilde g|_{M_s})^{-1}(1)$ is a clopen subset of $M_s$. Due to \eqref{eq1},
\[
{\rm cl}_{M(H^\infty)}(\Omega_{\varepsilon})\cap M_s\subset {\rm cl}_{M(H^\infty)}(S_{\Di}(f;\delta/2))\cap M_s\subset S_{M_s}(\hat f;\delta).
\]
Finally, since $x$ is a limit point of 
$S_{\Di}(f;\varepsilon)$, it belongs to ${\rm cl}_{M(H^\infty)}(\Omega_{\varepsilon})$ as well.
\end{proof}
 \begin{C}\label{cor1}
$M_s$ is totally disconnected. 
\end{C}
\begin{proof}
By the definition of the Gelfand topology, any open neighbourhood of $x\in M_s$ in $M_s$  contains an open neighbourhood  of the form
\[
\bigcap_{i=1}^n \{S_{M_s}(\hat f_i;\delta_i)\, :\,  \hat f_i(x)=0,\ \|f_i\|_\infty=1,\
\delta_i\in (0,1)\},  \quad n\in\mathbb N.
\]
In turn, each of the latter sets contains a clopen neighbourhood of $x$ by Theorem \ref{te1}. Thus, $M_s$ has the base of topology consisting of clopen sets, i.e., $M_s$ is totally disconnected.
\end{proof}

\begin{R}
{\rm In our arguments, we used the theorem of Su\'arez \cite[Th.~3.2]{S1} whose proof relies on the Carleson estimates \cite[Ch.~VIII,\,Th.~5.1]{Ga} and the fact that algebra $H^\infty$ is separating.  (In fact, the latter is not required as one can argue as in the proof of \cite[Th.~1.7]{Br2}.) Alternatively, here one can use Bishop's theorem \cite[Th.~1.1]{B}.}  
\end{R}

\sect{$H^2(M(H^\infty),\mathbb Z)=0$}
A commutative unital complex Banach algebra $A$ is said to
be {\em projective free} if every finitely generated projective
$A$-module is free.  The Novodvorski-Taylor theory asserts that the Gelfand transform $\ \hat{} : A\rightarrow C(M(A))$ determines an isomorphism between categories $P(A)$ of isomorphism classes of finitely generated projective $A$-modules and $Vect_{\Co}(M(A))$ of isomorphism classes of finite rank complex vector bundles on the maximal ideal space $M(A)$ of $A$, see \cite{No}, \cite[Th.~6.8,\, p.\,199]{Ta}. This results in the following statement:
\begin{Prop}\label{te2}
The following conditions are equivalent:
\begin{itemize}
\item[(1)]
$A$ is a projective free algebra:
\item[(2)]
$C(M(A))$ is a projective free algebra:
\item[(3)]
$M(A)$ is connected and each finite rank complex vector bundle on $M(A)$ is topologically trivial.
\end{itemize}
\end{Prop}
Since isomorphism classes of rank one complex vector bundles on $M(A)$ are in the one-to-one correspondence (determined by assigning to each bundle its first Chern class) with elements of the \v{C}ech cohomology group $H^2(M(A),\mathbb Z)$ (see, e.g., \cite{Hus}), projective freeness of $A$ implies that $H^2(M(A),\mathbb Z)=0$. It is known that $H^\infty$ is projective free \cite[Cor.~3.30]{Q}, \cite[Th.~1.5]{BS}. Thus, we get
\begin{C}\label{cor2}
$H^2(M(H^\infty),\mathbb Z)=0$.
\end{C}
\begin{R}
{\rm (1) In fact, from the projective freeness of $A$ follows also that even rational \v{C}ech cohomology groups $H^{2m}(M(A),\mathbb Q)=0$ for all $m\ge 2$. This is the consequence of some fundamental result of $K$-theory, see e.g., \cite{K}.\\
(2)  In \cite[Th.~1.5]{BS} projective freeness is established for algebras $H^\infty(U)$ for a large class of Riemann surfaces $U$. In this case, as in Corollary \ref{cor2}, we obtain
$H^2(M(H^\infty(U)),\mathbb Z)=0$. For instance, as such $U$ one can take an unbranched covering of an open bordered Riemann surface. The proof in \cite[Th.~1.5]{BS} is based on an analog of the Beurling-Lax-Halmos theorem established by the author in an earlier paper. For the sake of completeness, we place its version for $H^\infty$ in the Appendix.
 }
\end{R}
\sect{$M(H^\infty)$ is the Freudenthal compactification of $M_a$} 
Let $X$ be a semicompact Hausdorff space (i.e. every point of $X$ has arbitrarily small neighbourhoods with compact boundaries). The {\em Freudenthal compactification} $\gamma(X)$ of $X$ is the unique (up to homeomorphism) Hausdorff compactification of $X$\footnote{i.e., a compact Hausdorff space containing $X$ as an open dense subset} having the following properties
\begin{itemize}
\item[(a)]
$\gamma(X)\setminus X$ is zero-dimensionally embedded in $\gamma(X)$, i.e., any point in $\gamma(X)\setminus X$ has arbitrarily small neighbourhoods whose boundaries lie in $X$;
\item[(b)]
$\gamma(X)$ is maximal with respect to (a). That is, if  $c(X)$ is a Hausdorff compactification of $X$ such that $c(X)\setminus X$ is zero-dimensionally embedded in 
$c (X)$, then the identity map on $X$ has a continuous extension from $\gamma(X)$ to $c(X)$.
\end{itemize}

Let us mention some properties of $\gamma(X)$ (see also \cite{F,M,I,D,DM}):\smallskip

-- Any two disjoint closed subsets of $X$ with compact boundaries have disjoint closures in $\gamma(X)$;\smallskip

-- $\gamma(X)$ is a perfect compactification, i.e., for each $x\in\gamma(X)$ and each open neighbourhood $U$ of $x$ in $\gamma (X)$ set $U\cap X$ is not disjoint union of two open sets $V$ and $W$ such that $x\in {\rm cl}_{\gamma(X)}(U)\cap {\rm cl}_{\gamma(X)}(W)$. In fact, $\gamma(X)$ is the unique perfect compactification of $X$ in which $\gamma (X) \setminus X$ zero-dimensionally embeds;\smallskip

-- If $X$ is connected and locally connected, then so is $\gamma(X)$;\smallskip

-- Any homeomorphism between any two semicompact Hausdorff spaces extends to a homeomorphism between their Freudenthal compactifications.\medskip

Also, the Freudenthal compactification $\gamma(X)$ can be determined as follows: 

Let $\bar{C}_{\rm fin}(X)$  be closure in $C_b(X)$ (-\,the Banach algebra of bounded complex-valued continuous functions on $X$) of the algebra $C_{\rm fin}(X)$ of all functions $f\in C_b(X)$ for which there is a compact subset $K\subset X$ such that $f(X\setminus K)\subset\Co$ is finite. Then the maximal ideal space $M(\bar C_{\rm fin}(X))$ of $\bar C_{\rm fin}(X)$ is homeomorphic to $\gamma(X)$.

The main result of this section is 
\begin{Th}\label{te3}
$M(H^\infty)$ is homeomorphic to $\gamma(M_a)$.
\end{Th}
\begin{R} 
{\rm (1) Note that $M_a$ is locally compact and, hence, semicompact. In fact, the base of topology of $M_a$ consists of sets of the form $S_{M_a}(\hat B;\varepsilon):=\{x\in M_a\, :\, |\hat B(x)|<\varepsilon\}$, where $B$ is an interpolating Blaschke product, which for all sufficiently small $\varepsilon$ are relatively compact subsets of $M_a$ (see, e.g., \cite[Sect.~2.2]{Br2}). Therefore $\gamma(M_a)$ is well defined.

\noindent (2) Theorem \ref{te3} implies that $C(M(H^\infty))$ is isometrically isomorphic to $\bar{C}_{\rm fin}(M_a)$ (cf. Bishop \cite[Th.~1.1]{B}).
}
\end{R}
\begin{proof}
\begin{Lm}
Each function in $C_{\rm fin}(M_a)$ can be continuously extended to a function in $C(M(H^\infty))$ and the set of all such extensions separates points of $M(H^\infty)$. 
\end{Lm}
\begin{proof}
Let $f\in C_{\rm fin}(M_a)$ and $K\subset M_a$ be compact such that $f(M_a\setminus K)\subset\Co$ is finite.
We set $U=M(H^\infty)\setminus K$. Then $U$ is an open neighbourhood of $M_s$ and $U\cap\Di=(M_a\setminus K)\cap\Di$ is disjoint union of connected open sets. So, $f|_{U\cap\Di}$ is a bounded continuous function with finite range. In particular,
it is constant on each connected component of $U\cap\Di$, i.e., $f|_{U\cap\Di}\in H^\infty(U\cap\Di)$. Hence, applying \cite[Th.~3.2]{S1} we extend $f|_{U\cap\Di}$ continuously to a bounded function $f'\in C(U)$. Since $U\cap\Di$ is dense in $M_a\setminus K$, 
\[
f'(x)=f(x)\quad {\rm for\ all}\quad x\in M_a\setminus K.
\]
Function $\tilde f\in C(M(H^\infty)$ equals $f$ on $M_a$ and $f'$ on $M_s$ is the required extension of $f$.

Further,  for distinct points $x,y\in M_s$ due to the fact that ${\rm dim}M_s=0$ we can find open neighbourhoods $U_x$ and $U_y$ of $x$ and $y$ in $M(H^\infty)$ such that
\[
\begin{array}{l}
{\rm cl}(U_x)\cap {\rm cl} (U_y)=\emptyset,\quad {\rm cl}(U_x)\cap M_s=U_x\cap M_s,\quad  {\rm cl}(U_y)\cap M_s=U_y\cap M_s\quad {\rm and}\medskip\\
M_s=(U_x\cap M_s)\cup (U_y\cap M_s)
\end{array}
\]
(cf. \cite[Lm.~4.1]{Br2} for similar arguments).

Consider a function $g\in C\bigl({\rm cl}(U_x)\cup{\rm cl}(U_y)\bigr)$ equals $0$ on 
${\rm cl}(U_x)$ and $1$ on ${\rm cl}(U_y)$. Let $g_e\in C(M(H^\infty))$ be a continuous extension of $g$ (existing by the Tietze-Urysohn theorem). Then $g_e$ attains values $0$ and $1$ outside compact set
$K:=M(H^\infty)\setminus (U_x\cup U_y)\subset M_a$. By definition $f:=g_e|_{M_a}\in C_{\rm fin}(M_a)$ and its extension $\tilde f=g_e$ separates points $x$ and $y$, as required. 

The fact that extensions $\tilde f$ as above separate points $x\in M_s$ and $y\in M_a$ or distinct points $x,y\in M_a$ is obvious.
\end{proof}
Due to the lemma, algebra $\bar C_{\rm fin}(M_a)$ admits a continuous norm-preserving extension to $M(H^\infty)$. Since the former algebra is self-adjoint with respect to the complex conjugation, this extension coincides with $C(M(H^\infty))$ by the Stone-Weierstrass theorem. In particular, the maximal ideal space $M(\bar C_{\rm fin}(M_a))$ is homeomorphic to $M(H^\infty)$. On the other hand, it is homeomorphic to the Freudenthal compactification $\gamma(M_a)$ of $M_a$. This completes the proof of the theorem.
\end{proof}

\sect{Appendix: $H^\infty$ is a Projective Free Algebra}
We use that $A$
is projective free iff every nonzero square idempotent matrix with entries in $A$ is similar
(by an invertible matrix with entries in $A$) to a matrix of the form
$$
\textrm{diag}(I_k,0):=\left[ \begin{array}{cc} I_k & 0 \\ 0 & 0 \end{array} \right],\quad k\in\mathbb N,
$$
where $I_k$ is the identity $k\times k$ matrix  (see \cite[Prop.~2.6]{Co}).

By $H_n^\infty$ we denote the $H^\infty$-module consisting of
columns $(f_1,\dots, f_n)$, $f_i\in H^\infty$. An
$H^\infty$-invariant subspace of $H_n^\infty$ is called a {\em
  submodule}. The following result can be deduced from the classical Beurling-Lax-Halmos theorem 
  (see, e.g., \cite[ p.\,1025]{To}).\smallskip

\noindent (BLH) {\em Let $M\subset H_n^\infty$ be a nonzero weak$^*$ closed submodule. Then $M=H\cdot H_k^\infty$ for some $1\le k\le n$, where $H$ is a $n\times k$ left unimodular matrix with entries in $H^\infty$, i.e., $(H(e^{it}))^*\cdot H(e^{it})=I_k$ for a.e. $t\in [0,2\pi]$.}
 
 \begin{Theo}[\cite{Q, BS}]
 $H^\infty$ is a projective free algebra.
 \end{Theo}
\begin{proof} (We follow the arguments in  \cite{BS}.) Let $F$ be a nontrivial
  idempotent of size $n\times n$ with entries in $H^\infty$. By 
  definition, $F$ determines a weak$^*$ continuous linear operator
  $H_n^\infty\rightarrow H_n^\infty$ such that $M_1:=\textrm{im}(F)=\ker(I_n-F)$. Hence, $M_1\subset H_n^\infty$ is a weak$^*$ closed submodule.  According to (BLH)  $M_1=H_1\cdot H_k^\infty$, where $H_1$ is a $n\times k$ left unimodular matrix with entries in $H^\infty$. In particular, $\hat H_1(\xi)$ is left invertible at any point $\xi$ of the \v{S}ilov boundary of $M(H^\infty)$.
 Since $F$ has the same rank at each point
  of $M(H^\infty)$ (as $M(H^\infty)$ is connected),
  the invertibility of $\hat H_1(\xi)$ implies $k={\rm rank}(F)$. Thus,
  $\hat H_1(\xi)$ is left invertible for all $\xi\in M(H^\infty)$.

 Similarly, $M_2:=\ker(F)=\textrm{im}(I-F)\subset H_n^\infty$ is a weak$^*$ closed submodule. So,  $M_2=H_2\cdot H_{n-k}^\infty$, where $H_2$ is a $n\times (n- k)$ matrix with entries in $H^\infty$ such that $\hat H_2$ is left invertible at each point of $M(H^\infty)$.
  From the fact $M_1\cap M_2=\{0\}$ follows that the $n\times n$ matrix
  $H=(H_1, H_2)$  with entries in $H^\infty$  is invertible and
  $H^{-1}$ has entries in $H^\infty$ as well. Moreover, $H^{-1}\cdot F\cdot
  H=\textrm{diag}(I_k,0)$. 
  \end{proof}

{}


\begin{thebibliography}{}

\bibitem[B]{B}
C. J. Bishop, Some  Characterizations of $C(\mathcal M)$, Proc. Amer. Math. Soc. {\bf 24} (9) (1996), 2695--2701.
\bibitem[BS]{BS}
A.~Brudnyi and A.~Sasane, Sufficient conditions for the projective freeness of Banach algebras, J. Funct. Anal. {\bf 257} (2009), 4003--4014.

\bibitem[Br1]{Br1}
A.~Brudnyi, Topology of the maximal ideal space of $H^\infty$, J. Funct. Anal. {\bf 189} (2002), 21--52. 

\bibitem[Br2]{Br2}
A.~Brudnyi,  Stein-like theory for Banach-valued holomorphic functions on the maximal ideal space of $H^\infty $, Invent. math. {\bf 193} (2013), 187--227.

\bibitem[C]{C}
L.~Carleson, Interpolations by bounded analytic functions and the corona theorem, Ann. of Math. {\bf 76}
(1962), 547--559.

\bibitem[Co]{Co}
P. M. Cohn, Free Rings and their Relations, London Math. Soc. Monographs {\bf 19}, London, 1985.

\bibitem[D]{D}
R. F.~Dickman Jr., Some characterizations of the Freudenthal compactification of a semicompact space, Proc. Amer. Math. Soc. {\bf 19} (1968) 631--633.

\bibitem[DM]{DM}
R. F.~Dickman, R. A.~McCoy, The Freudenthal Compactification, Dissertationes Math., vol. CCLXII, Polish Acad. Sci., Warszawa, 1988.

\bibitem[F]{F}
H. Freudenthal, Neuaufbau der Endentheorie, Ann. Math. {\bf 43} (1942), 261--279.

\bibitem[Ga]{Ga}
J.~B.~Garnett, Bounded analytic functions, Academic Press, New York, 1981.

\bibitem[GN]{GN}
J.~B.~Garnett and A. Nicolau, Interpolating Blaschke products generate $H^\infty$, Pacific J. Math. {\bf 173} no. 2 (1996), 501--510.

\bibitem[H]{H}
K.~Hoffman, Bounded analytic functions and Gleason parts, Ann. of Math. {\bf 86} (1967), 74--111.

\bibitem[Hus]{Hus}
D.~Husemoller, Fibre bundles, Springer-Verlag, New York, 1994.


\bibitem[I]{I}
J. R.~Isbell, Uniform spaces, Math. Surveys, no. 12, Amer. Math. Soc., Providence, R. I., 1964.

\bibitem[K]{K}
M. Karoubi,  K-theory, an introduction, Grundlehren der Math. Wissen.
{\bf 226}, Springer, Berlin, 1978.

\bibitem[M]{M}
K. Morita, On bicompactifications of semibicompact spaces, Sci. Rep. Tokyo
Bunrika Daigaku, Sect. {\bf 4} No. 94 (1952), 222--229.

  

\bibitem[N]{N}
K. Nagami,  Dimension Theory, Academic Press, New York, 1970.

\bibitem[No]{No}
E. Novodvorskii, Certain homotopical invariants of spaces of maximal ideals, Mat.
Sb. {\bf 1} (1967), 487--494.

\bibitem[Q]{Q}
A.~Quadrat, The fractional representation approach to synthesis problems: an algebraic
analysis viewpoint. I. (Weakly) doubly coprime factorizations, SIAM J. Control
Optim. {\bf 42} (2003), 266--299.

\bibitem[S1]{S1}
D.~Su\'{a}rez, \v{C}ech cohomology and covering dimension for the $H^\infty$ maximal ideal space, J. Funct. Anal. {\bf 123} (1994),
233--263.

\bibitem[S2]{S2}
D.~Su\'{a}rez, Trivial Gleason parts and the topological stable rank of $H^\infty$, Amer. J. Math. {\bf 118} (1996), 879--904.

\bibitem[Ta]{Ta}
J. L. Taylor, Topological invariants of the maximal ideal space of a Banach algebra, Adv. Math. {\bf 19}  (1976), 149--206.

\bibitem[T]{T}
S.~Treil,  The stable rank of $H^\infty$ equals $1$, J. Funct. Anal. {\bf 109} (1992), 130--154.

\bibitem[To]{To}
V. Tolokonnikov, Extension problem to an invertible matrix, Proc.  Amer. Math. Soc. {\bf 117} (4) (1993), 1023--1030.


\end{thebibliography}
\end{document}